\documentclass[a4article,10pt]{amsart}
\usepackage[english]{babel}
\usepackage{amsmath,amsthm}
\usepackage{amsfonts}
\usepackage{comment}
\usepackage[shortlabels]{enumitem}
\usepackage{graphicx}
\usepackage[dvipsnames,usenames]{color}
\usepackage[dvipsnames]{xcolor}
\usepackage[pagebackref,hypertexnames=false,colorlinks=true,
            linkcolor=NavyBlue,
            citecolor=NavyBlue,
            urlcolor=NavyBlue]{hyperref} 
\usepackage{cleveref}
\usepackage{subcaption}
\usepackage{color}
\usepackage[all]{xy}
\usepackage{thmtools, thm-restate}
\usepackage{overpic}
\usepackage{amssymb,tikz}
\usepackage{quiver}
\usetikzlibrary{knots,decorations.pathreplacing}
\usepackage{multicol}
\usetikzlibrary{arrows.meta, positioning, calc}

\newcommand{\red}{\textcolor{red}}

\newcommand{\DM}[1]{\red{[DM] #1}}

\newcommand{\Z}{\mathbb{Z}}

\newcommand{\C}{\mathbb{C}}

\newcommand{\CP}{\mathbb{CP}}

\usepackage[top=1.23in, bottom=1.22in, left=1.23in, right=1.23in]{geometry}

\newtheorem{thm}{Theorem}[section]
\newtheorem{cor}[thm]{Corollary}
\newtheorem{lem}[thm]{Lemma}
\newtheorem{prop}[thm]{Proposition}

\newtheorem{claim}{Claim}
\newtheorem*{case}{Case}

\theoremstyle{definition}
\newtheorem{defn}[thm]{Definition}

\newtheorem{ques}[thm]{Question}

\theoremstyle{remark}
\newtheorem{rem}[thm]{Remark}

\numberwithin{equation}{section}

\newcommand{\spin}{\ifmmode{\rm Spin}\else{${\rm spin}$\ }\fi}
\newcommand{\spinc}{\ifmmode{{\rm Spin}^c}\else{${\rm spin}^c$}\fi}

\DeclareMathOperator{\tors}{tors}

\newcommand{\wt}[1]{\widetilde{ #1 }}

\newcommand{\plumb}{\entrymodifiers={+[o][F-]} \xymatrix@C=8pt}
\newcommand{\el}{\ar@{-}[r]}
\newcommand{\ed}{\ar@{..}[r] }

\DeclareMathOperator{\supp}{Supp}

\begin{document}

\title{Smooth Realizations of Line Configurations}%

\author{Paolo Aceto}%
\address {Université de Lille}
\email{paoloaceto@gmail.com }

\author{Duncan McCoy}%
\address {Universit\'{e} du Qu\'{e}bec \`{a} Montr\'{e}al}
\email{mc\_coy.duncan@uqam.ca}

\author{JungHwan Park}%
\address {Korea Advanced Institute of Science and Technology}
\email{jungpark0817@kaist.ac.kr}
\date{\today}%

\begin{abstract}
We study the problem of realizing line configurations as collections of 2-spheres smoothly embedded in the complex projective plane. Building upon prior work by Ruberman and Starkston on topological realizations, we establish a stronger obstruction in the smooth category. Our proof relies on lattice-theoretic arguments based on Donaldson's diagonalization theorem.
\end{abstract}
\maketitle

\section{Introduction}
The study of line arrangements and configurations is a classical topic in geometry and combinatorics (see, e.g., \cite{BS:2005-1, BS:2013, BFG:2017, Ruberman-Starkston:2019, Golla:2024-1, KST:2024-1, Urzua:2022-1, Aceto-Golla:2026} for some recent developments and open problems). 
A \emph{combinatorial line arrangement} is a pair of finite sets containing lines and points and equipped with an incidence relation such that any two distinct lines are incident to a unique common point. A \emph{combinatorial $(n_k)$-configuration}, or simply an \emph{$(n_k)$-configuration}, is a combinatorial line arrangement consisting of $n$ lines and a distinguished subset $\mathcal{K}$ of $n$ points such that each point of $\mathcal{K}$ lies on exactly $k$ lines, and each line contains exactly $k$ points of $\mathcal{K}$. To avoid trivialities, we always assume $k\geq 3$.
By counting incidences on one fixed line, it is easy to see that any $(n_k)$-configuration must satisfy
\begin{equation}\label{eq:n_k_inequality}
n \geq k^2-k+1.
\end{equation}
Although an $(n_k)$-configuration need not exist for every value of $n$ satisfying this inequality, there are infinitely many examples with $n=k^2-k+1$. Such a configuration is called a \emph{finite projective plane}. For instance, if $q$ is a prime power, the projective plane $\mathbb{P}^2(\mathbb{F}_q)$ gives an example with $k=q+1$ and $n=q^2+q+1$. It is known that not all finite projective planes arise as projective spaces in this way~\cite{Batten:1997}.

A fundamental question in algebraic and incidence geometry is which
$(n_k)$-configurations can be realized by lines in $\mathbb{RP}^2$ or
$\CP^2$; see~\cite{Grunbaum:2009} for the history and state of the art of
this question, as well as related problems. In this article, we focus on $\CP^2$. We say that an $(n_k)$-configuration is \emph{geometrically $\C$-realizable} if it can be realized by an arrangement of complex projective lines in $\CP^2$ with the prescribed incidence relation.

Furthermore, one can change the category and consider realizations in $\CP^2$ by locally flat embedded 2-spheres. A \emph{topological $\C$-realization} is such a realization with the lines replaced by locally flat embedded 2-spheres in $\CP^2$ representing the line class. This problem was recently pursued by Ruberman and Starkston~\cite{Ruberman-Starkston:2019}, who used the Atiyah--Singer $G$-signature theorem~\cite{Atiyah-Singer:1968, Gordon:1986} to prove that finite projective planes of the form $\mathbb{P}^2(\mathbb{F}_q)$ are not topologically $\C$-realizable.

Moreover, they showed that there is no topological $\C$-realization of the unique $(14_4)$-configuration. This is the first case of interest beyond finite projective planes not accounted for by the geometrically realizable $(8_3)$-configuration. More recently, the first author and Golla~\cite{Aceto-Golla:2026} improved these results by proving that if an $(n_k)$-configuration is topologically $\C$-realizable, then $n\geq k^2-5$. In this article, we improve this bound in the smooth category. Following~\cite{Ruberman-Starkston:2019}, we call such realizations \emph{smooth $\C$-realizations}; see Section~\ref{sec:definitions} for precise definitions.

\begin{thm}\label{thm:main}
If there exists a smooth $\C$-realization of an $(n_k)$-configuration with $k \geq 4$, then
\[
n \geq k^2.
\]
\end{thm}


We remark that one particularly interesting case ruled out by Theorem~\ref{thm:main} is that of $(15_4)$-configurations. To the best of our knowledge, it is not known whether any $(16_4)$-configuration admits a smooth $\C$-realization, whereas the unique $(17_4)$-configuration is known to admit a smooth $\C$-realization; see~\cite{BGB:2009} and~\cite[Theorem~1.3]{Ruberman-Starkston:2019}. Thus, combining the known results with Theorem~\ref{thm:main}, we obtain the following immediate corollary.

\begin{cor}\label{cor:small-parameters}
Suppose $n\leq k^2-k+3$. Then the following are equivalent:
\begin{enumerate}[label=\textup{(\arabic*)}, ref=\arabic*]
    \item\label{item:smooth-small} there exists a smooth $\C$-realization of an $(n_k)$-configuration;
\item\label{item:geometric-small} there exists a geometric $\C$-realization of an $(n_k)$-configuration;
\item\label{item:parameters-small} $k=3$ and $n\in\{8,9\}$.
\end{enumerate}
\end{cor}

\begin{proof}
Since every geometric $\C$-realization is a smooth $\C$-realization, \textup{(\ref{item:geometric-small})} implies \textup{(\ref{item:smooth-small})}. Assume \textup{(\ref{item:smooth-small})}. If $k\geq4$, then
\[
n\leq k^2-k+3<k^2,
\]
contradicting Theorem~\ref{thm:main}. Hence $k=3$. By \eqref{eq:n_k_inequality}, we have
\[
n\geq k^2-k+1=7,
\]
while the hypothesis gives $n\leq9$. The case $n=7$ is the finite projective plane of order two, which is not topologically $\C$-realizable by Ruberman--Starkston~\cite[Theorem~1.1]{Ruberman-Starkston:2019}. Therefore \textup{(\ref{item:parameters-small})} holds. Finally, assume \textup{(\ref{item:parameters-small})}. The classical
Möbius--Kantor configuration gives a geometrically $\C$-realizable
$(8_3)$-configuration, and the classical Pappus configuration gives a
geometrically $\C$-realizable $(9_3)$-configuration; see, for example,
\cite[Section~2.1]{Grunbaum:2009}. Hence
\textup{(\ref{item:geometric-small})} holds.
\end{proof}






Clearly, the bound in Theorem~\ref{thm:main} also applies to
\emph{geometric $\C$-realizations}. In the geometric setting, however,
inequalities for complex line arrangements give stronger bounds. More
precisely, any geometric $\C$-realization of an $(n_k)$-configuration with
$k\geq4$ satisfies
\[
n\geq
\begin{cases}
16, & k=4,\\[2mm]
\left\lceil \dfrac{3}{2}k^2-3k+3\right\rceil, & k\geq5.
\end{cases}
\]
Although this may be known to experts, we include the proof for completeness;
see Proposition~\ref{prop:hirzebruch}. The proof uses Hirzebruch's
inequality~\cite{Hirzebruch:1983}, together with the classical classification
of its equality case~\cite[Section~5.2]{Barthel-Hirzebruch-Hoefer:1987-weighted},
for $k=4$, and Langer's refinement~\cite{Langer:2003} for $k\geq5$. It is noteworthy that, for $k=4$, this geometric bound coincides with the smooth bound of Theorem~\ref{thm:main}.



The proof of Theorem~\ref{thm:main} is divided into two steps, both based on
lattice-theoretic arguments relying on Donaldson's diagonalization
theorem~\cite{Donaldson:1987-1}. In the first step, carried out in
Section~\ref{sec:ngeqk^2-1}, we prove that $n\geq k^2-1$. Starting from a
smooth $\C$-realization, we blow up negatively at the points of $\mathcal K$
and then perform a controlled sequence of positive blow-ups and negative
blow-downs. This produces a positive-definite 4-manifold, so Donaldson's
theorem forces its intersection lattice to be standard. The resulting
configuration of spheres determines a weighted graph, and the lattice
combinatorics of this graph obstructs the required standard embedding.

The second step, carried out in Section~\ref{sec:nneqk^2-1}, rules out the
remaining borderline case $n=k^2-1$; this is the more delicate part of the
argument. In this case, after perturbing the remaining intersections to
double points, there are
\[
N=\frac{n(k-2)}2
\]
positive blow-up centers. The final transforms of the original spheres are
pairwise disjoint $(-1)$-spheres. Blowing them down gives a
positive-definite 4-manifold whose associated lattice would have to be
standard, again by Donaldson's theorem. We then construct a unit vector that
splits off a standard summand, together with $N$ square-three vectors in its
orthogonal complement. A combinatorial argument in this standard lattice
gives equations that are incompatible with the incidence structure for
$k\geq5$. The case $k=4$ requires a slightly stronger argument, using special
combinatorics of $(15_4)$-configurations.

We end the introduction with two natural questions. While there exist line arrangements that can be realized smoothly but not geometrically (see, e.g., \cite{Ruberman-Starkston:2019}), to the best of our knowledge there is no known example of an $(n_k)$-configuration that admits a smooth $\C$-realization but not a geometric one. Moreover, in light of the well-known discrepancy between the topological and smooth categories in dimension four~\cite{Freedman:1982-1, Donaldson:1983-1}, we ask:

\begin{ques}
Does there exist an $(n_k)$-configuration that admits a smooth
$\C$-realization but not a geometric one? Does there exist one that admits a
topological $\C$-realization but not a smooth one?
\end{ques}

\subsection*{Acknowledgments} 
PA is supported by the European Union’s Horizon 2020 research and innovation programme under the Marie Sk\l{}odowska-Curie action, Grant No.\ 101030083 LDTSing.
DM is partly supported by an NSERC Discovery grant and a Canada Research Chair. 
JP is partially supported by Samsung Science and Technology Foundation (SSTF-BA2102-02) and the NRF grant RS-2025-00542968.

\section{Line arrangements and configurations}\label{sec:definitions}
First, we recall the relevant definitions (see also \cite{Ruberman-Starkston:2019}).
\begin{defn}A \emph{combinatorial line arrangement} $\mathcal{A}$ is a pair of finite sets $\mathcal{L}$ and $\mathcal{P}$, consisting of lines and points, respectively, equipped with a function $i \colon \mathcal{L} \times \mathcal{P} \rightarrow \{0,1\}$, the \emph{incidence relation}, such that for every pair of distinct lines $L,L'\in \mathcal{L}$ there exists a unique point $P\in \mathcal{P}$ such that $i(L,P)=i(L',P)=1$. 
\end{defn}

We say that a line $L$ contains a point $P$ or that $P$ lies on $L$ if $i(L,P)=1$.


\begin{defn}
A \emph{smooth $\C$-realization} of a combinatorial line arrangement $\mathcal{A} = (\mathcal{L}, \mathcal{P})$ is a collection of $|\mathcal{L}|$ smoothly embedded 2-spheres in $\CP^2$ such that:
\begin{itemize}
    \item each sphere represents the homology class of a complex projective line;
    \item all intersections are positive, transverse, and prescribed by the incidence relation of~$\mathcal{A}$.
\end{itemize}
Similarly, a \emph{topological $\C$-realization} is defined by simply replacing smooth embeddings with locally flat topological embeddings.
\end{defn}

\begin{defn} An \emph{$(n_k)$-configuration} is a combinatorial line arrangement $\mathcal{A}=(\mathcal{L},\mathcal{P})$ where $|\mathcal{L}|=n$ and $\mathcal{P}$ contains a subset $\mathcal{K}$ of $n$ points such that each point in $\mathcal{K}$ lies on exactly $k$ lines and each line contains exactly $k$ points from $\mathcal{K}$. 
\end{defn}
As mentioned in the introduction, every $(n_k)$-configuration satisfies
\begin{equation*}
n \geq  k^2 -k +1.
\end{equation*}
Given a fixed choice of line $L$ in the configuration, each of the $k$ points from $\mathcal{K}$ on $L$ provides an intersection with $k-1$ other lines in the configuration. Since $L$ intersects every other line in a unique point, this gives at least $k^2-k$ lines intersecting $L$, and hence the inequality. For future reference, we state the topological bound already mentioned in the introduction.

\begin{thm}[{\cite{Aceto-Golla:2026}}]\label{thm:RS}
If there exists a topological $\mathbb{C}$-realization of an $(n_k)$-configuration with $k \geq 4$, then
\[
n \geq k^2 - 5.
\]
In particular, the same inequality holds for any $(n_k)$-configuration admitting a smooth $\mathbb{C}$-realization.
\end{thm}

\section{Lattices}
 Throughout, we consider $\mathbb{Z}^N$ equipped with the standard positive-definite inner product, and we fix an orthonormal basis $\{e_1, \dots, e_N\}$ for $\mathbb{Z}^N$. Given this choice of basis, if $v \in \mathbb{Z}^N$ is a vector, we define the \emph{support} of $v$ to be
\[
\supp(v) := \{1 \leq i \leq N \mid e_i \cdot v \neq 0\}.
\]
Similarly, if $S \subset \mathbb{Z}^N$ is a set of vectors, we define the \emph{support} of $S$ to be
\[
\supp(S) := \bigcup_{v \in S} \supp(v).
\]

 \begin{defn}
Let $\Gamma$ be a finite simple weighted graph. An \emph{embedding} of $\Gamma$ into $\mathbb{Z}^N$ is a map
\[
\varphi \colon V(\Gamma) \rightarrow \mathbb{Z}^N
\]
such that
\[
|\varphi(u) \cdot \varphi(v)| =
\begin{cases}
    w(u) & \text{if } u = v, \\
    1   & \text{if } u \neq v \text{ are adjacent}, \\
    0    & \text{otherwise},
\end{cases}
\]
for all vertices $u, v \in V(\Gamma)$.
\end{defn}

Given a choice of embedding $\varphi$, we will simplify notation by identifying each vertex with its image in $\mathbb{Z}^N$. Furthermore, we will abuse notation by treating any graph $G$ interchangeably with its vertex set $V(G)$ when no confusion arises.
We focus on graphs $\Gamma$ satisfying the following somewhat arcane collection of hypotheses, the motivation for which will become clear later in the proof.

\begin{defn}\label{def:nice_graph}
Let $\Gamma$ be a simple weighted graph. We say that $\Gamma$ is \emph{nice} if every vertex has weight two or three and $\Gamma$ can be written as a union
\[
\Gamma = \bigcup_{\alpha=1}^n G_\alpha
\]
of subgraphs satisfying the following conditions:
\begin{enumerate}[(I)]
    \item\label{cond:Gi_complete} Two vertices $u$ and $v$ are adjacent if and only if there is some $G_\alpha$ containing both $u$ and $v$;
    \item\label{cond:linking_distinct_components} For all $\alpha \neq \beta$, either $G_\alpha$ and $G_\beta$ are disjoint, or $G_\alpha \cap G_\beta$ consists of a unique vertex which has weight three;
    \item\label{cond:weight_three_vertices} Every vertex of weight three is contained in precisely two of the $G_\alpha$.
\end{enumerate}
\end{defn}

\begin{rem}
Condition~\ref{cond:Gi_complete} implies that each $G_\alpha$ is a complete graph, and Condition~\ref{cond:linking_distinct_components} implies that every edge of $\Gamma$ is contained in precisely one $G_\alpha$.
\end{rem}

\begin{lem}\label{lem:vertex_pairs}
Let $\Gamma = \bigcup_{\alpha=1}^n G_\alpha$ be a nice graph such that each $G_\alpha$ contains at least three vertices, including at least one vertex of weight two and at least one vertex of weight three. Then, for any embedding of $\Gamma$ into $\mathbb{Z}^N$, and for any pair of distinct vertices $u$ and $v$, we have
\[
|\supp(u) \cap \supp(v)| = |u \cdot v| =
\begin{cases}
    1 & \text{if $u$ and $v$ are adjacent}, \\
    0 & \text{otherwise}.
\end{cases}
\]
\end{lem}

\begin{proof}
    Since all vertices have weight two or three, we necessarily have $|v \cdot e_j| \leq 1$ for all vertices $v$ and all basis vectors $e_j$. This implies that $v \cdot v = |\supp(v)|$ for every vertex $v$, and that for any pair of vertices $u$ and $v$ we have
\begin{equation}\label{eq:parity}
    u \cdot v \equiv |\supp(u) \cap \supp(v)| \pmod{2}
\end{equation}
and
\begin{equation}\label{eq:max_size}
    |u \cdot v| \leq |\supp(u) \cap \supp(v)| \leq \min\{|\supp(u)|, |\supp(v)|\}.
\end{equation}

    \begin{case}
        If $u,v$ are adjacent, then \[|\supp(u)\cap \supp(v)|=1.\]
    \end{case}
    \begin{proof}[Proof of Case]
        Together \eqref{eq:parity} and \eqref{eq:max_size} imply that if $u$ and $v$ are adjacent and at least one of $u$ or $v$ has weight two, then $|\supp(u)\cap \supp(v)|=1$. Thus it suffices to consider the case that both $u$ and $v$ have weight three. Using \eqref{eq:parity} and \eqref{eq:max_size}, we see that it suffices to rule out the possibility that $|\supp(u)\cap \supp(v)|=3$. However, this would mean that $\supp(u)=\supp(v)$ and, consequently, that $u\equiv v \pmod{2}$. Since $u$ and $v$ are adjacent, Condition~\ref{cond:Gi_complete} implies that there exists a subgraph $G_\alpha$ containing both $u$ and $v$. Condition~\ref{cond:weight_three_vertices} implies that there is another index $\beta \neq \alpha$ such that $u \in G_\alpha \cap G_\beta$. However, by Condition~\ref{cond:linking_distinct_components}, we must have $v \not\in G_\beta$. Thus, we may choose a vertex $z$ of weight two in $G_\beta$; such a vertex is adjacent to $u$ but not to $v$. Since $z \cdot u \not\equiv z \cdot v \pmod{2}$, this implies $u \not\equiv v \pmod{2}$, completing the argument.
    \end{proof}

    \begin{case}
        If $u,v$ are not adjacent and both have weight two, then \[|\supp(u)\cap \supp(v)|=0.\]
    \end{case}
    \begin{proof}[Proof of Case]
    Using \eqref{eq:parity} and \eqref{eq:max_size}, we see that it suffices to rule out the possibility that $|\supp(u) \cap \supp(v)| = 2$. However, this would imply that $\supp(u) = \supp(v)$ and, consequently, that $u \equiv v \pmod{2}$. To rule out this possibility, it suffices to exhibit a vertex $z$ that is adjacent to exactly one of $u$ or $v$, since this would yield $z \cdot u \not\equiv z \cdot v \pmod{2}$. We now construct such a vertex $z$ as follows. Since $u$ and $v$ are not adjacent, Condition~\ref{cond:Gi_complete} implies that there exists a subgraph $G_\alpha$ containing $u$ but not $v$. If $G_\alpha$ contains two weight-two vertices, then by Condition~\ref{cond:linking_distinct_components}, we may take $z$ to be another such vertex in $G_\alpha$. If $u$ is the unique vertex of weight two in $G_\alpha$, then the hypotheses of the lemma imply that $G_\alpha$ contains at least two vertices of weight three, say $v_1$ and $v_2$. However, again using Conditions~\ref{cond:Gi_complete} and~\ref{cond:linking_distinct_components}, we see that at most one of $v_1$ and $v_2$ can be adjacent to $v$. Thus, we take $z$ to be one of these, either $v_1$ or $v_2$.
    \end{proof}
    \begin{case}
        If $u,v$ are not adjacent and $u$ has weight two and $v$ has weight three, then 
        \[|\supp(u)\cap \supp(v)|=0.\]
    \end{case}
    \begin{proof}[Proof of Case]
    
    Since $u$ and $v$ are not adjacent, it follows that they do not lie in the same $G_\alpha$. By Conditions~\ref{cond:Gi_complete} and~\ref{cond:weight_three_vertices} there are distinct indices $\alpha,\beta$ such that $u\notin G_\alpha \cup G_\beta$ and $v\in G_\alpha \cap G_\beta$. By choosing a vertex of weight two from each of $G_\alpha, G_\beta$, we see that $u$ and $v$ are contained in an induced subgraph of the form

         
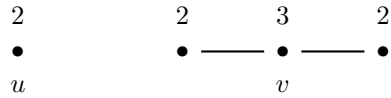
\[\begin{tikzcd}[row sep=0.1em]
	2& & 2 & 3 & 2 \\
	\bullet & &\bullet &  \bullet & \bullet\\
        u& & & v&
        \arrow[thick, no head, from=2-4, to=2-5]
	\arrow[thick, no head, from=2-3, to=2-4]
\end{tikzcd}\]
Using the first two cases of the lemma, notably, the fact that all vertices of weight two must have disjoint supports, we see that $|\supp(u) \cap \supp(v)| \leq 1$. This, combined with the parity condition~\eqref{eq:parity}, completes the proof of the case.\end{proof}

    \begin{case}
        If $u,v$ are not adjacent and $u$ and $v$ both have weight three, then
        \[|\supp(u)\cap \supp(v)|=0.\]
    \end{case}
    \begin{proof}[Proof of Case]
    By Conditions~\ref{cond:Gi_complete} and~\ref{cond:weight_three_vertices}, we see that there are distinct indices $\alpha, \beta, \gamma, \delta$ such that
\[
u\in G_\alpha \cap G_\beta \qquad \text{ and }\qquad v\in G_\gamma \cap G_\delta.
\]
By choosing a vertex of weight two from each of $G_\alpha, G_\beta, G_\gamma$, and $G_\delta$ we see that $u$ and $v$ are contained in an induced subgraph of the form
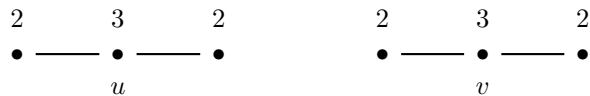
\[\begin{tikzcd}[row sep=0.1em]
	2&3 & 2 & &2& 3 & 2 \\
	\bullet &\bullet &\bullet & &\bullet &  \bullet & \bullet\\
        &u && & & v&
        \arrow[thick, no head, from=2-1, to=2-2]
	\arrow[thick, no head, from=2-2, to=2-3]
    \arrow[thick, no head, from=2-5, to=2-6]
	\arrow[thick, no head, from=2-7, to=2-6]
\end{tikzcd}\]
Using the first two cases of the lemma, we see that $|\supp(u) \cap \supp(v)| \leq 1$.
Again, together with the parity condition~\eqref{eq:parity}, this completes the proof of the case.
\end{proof}
Combining all the cases concludes the proof.
\end{proof}


\begin{lem}\label{lem:2-graph_clique}
Let $\Gamma=\cup_{\alpha=1}^n G_\alpha$ be a nice graph such that each $G_\alpha$ contains at least one vertex of weight three. Then, for any embedding of $\Gamma$ into $\Z^N$ and any $G_\alpha$, we have
\[
\left|\bigcap_{j=1}^k \supp (u_j)\right|=1,
\]
where $u_1, \dots, u_k$ are $k\geq 2$ distinct weight two vertices in $G_\alpha$.
\end{lem}
\begin{proof}
By Lemma~\ref{lem:vertex_pairs}, any pair of distinct $i$ and $j$ satisfy $|\supp(u_i) \cap \supp(u_j)| = 1$. Thus we have already established the lemma for $k=2$.
For $k=3$, we proceed by contradiction. Suppose that \[\supp(u_1)\cap \supp(u_2) \cap \supp(u_3)=\emptyset.\]
This implies that
\[
\supp(u_1) \cap \supp(u_2),\qquad \supp(u_1) \cap \supp(u_3), \qquad \text{ and } \qquad \supp(u_2) \cap \supp(u_3)
\]
are all distinct singletons. By the inclusion-exclusion principle, we have
\[
|\supp(u_1) \cup \supp(u_2) \cup \supp(u_3)| = 3.
\]
It follows that every basis vector
\[
e_i \in \supp(u_1) \cup \supp(u_2) \cup \supp(u_3)
\]
is contained in precisely two of $\supp(u_1)$, $\supp(u_2)$, and $\supp(u_3)$. In particular, $u_1+u_2+u_3 \equiv 0 \pmod 2 $. This is absurd, since the hypotheses of the lemma guarantee the existence of a vertex (of weight three) that is adjacent to all three of $u_1, u_2, u_3$. Thus, the conclusion holds for $k = 3$ as well.

 For $k\geq 4$, we proceed inductively. By the inductive assumption and since $k - 1 \geq 3$, we have
\[
\supp(u_1) \cap \dots \cap \supp(u_{k-1}) = \supp(u_2) \cap \supp(u_3)
\]
and
\[
\supp(u_2) \cap \supp(u_3) = \supp(u_2) \cap \dots \cap \supp(u_k).
\]
 Taking the intersection of the left-hand side of the first equality and the right-hand side of the second equality yields
\[
\left|\bigcap_{j=1}^k \supp(u_j)\right| = |\supp(u_2) \cap \supp(u_3)| = 1,
\]
which concludes the proof.
 \end{proof}

\begin{lem}\label{lem:general_clique}
Let $\Gamma = \bigcup_{\alpha=1}^n G_\alpha$ be a nice graph such that each $G_\alpha$ contains at least four vertices, including at least one vertex of weight two and at least two vertices of weight three. Then, for any embedding of $\Gamma$ into $\Z^N$ and any $G_\alpha$, we have
 \[
 \left|\bigcap_{v\in G_\alpha} \supp (v)\right|=1.
 \]
\end{lem}

    \begin{proof}
        Write the vertices of $G_\alpha$ as $u_1, \dots, u_k, v_1, \dots, v_\ell$, where $u_1, \dots, u_k$ all have weight two and $v_1, \dots, v_\ell$ have weight three. By hypothesis, we have $k \geq 1$, $\ell \geq 2$, and $k + \ell \geq 4$. We split the analysis into two cases.

    \begin{case}
         $k\geq 2$.
    \end{case}
    \begin{proof}[Proof of Case]
We assume $k \geq 2$ and prove the desired statement. By
Lemma~\ref{lem:2-graph_clique}, we have
\[
\left|\bigcap_{i=1}^k \supp(u_i)\right|
=
|\supp(u_1)\cap\supp(u_2)|
=
1.
\]
Thus, it is enough to show that, for every $v_j$,
\[
|\supp(u_1)\cap\supp(u_2)\cap\supp(v_j)|=1.
\]

Consider $v_j$, an arbitrary vertex of weight three. Let $v_{j'} \neq v_j$ be another vertex of weight three in $G_\alpha$. Then, by Conditions~\ref{cond:linking_distinct_components} and~\ref{cond:weight_three_vertices}, there exists $\beta \neq \alpha$ such that $v_j \in G_\beta$ but $v_{j'} \notin G_\beta$. By taking $u_1, u_2, v_j, v_{j'}$, and a vertex of weight two in $G_\beta$, we obtain an induced subgraph of the form:

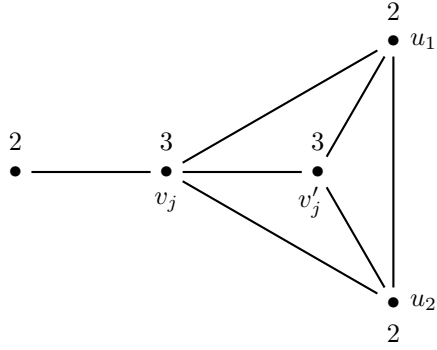
\[\begin{tikzpicture}[xscale=1,yscale=1,baseline={(0,0)}]

\node at (-2, .4) {$2$};
\node at (0, .4) {$3$};
\node at (3, 1.73+.4) {$2$};
\node at (3, -1.73-.4) {$2$};
\node at (2, .4) {$3$};

\node at (0, -.4) {$v_j$};
\node at (1.9, -.4) {$v_j'$};
\node at (3.4, 1.73) {$u_1$};
\node at (3.4, -1.73) {$u_2$};

\node (A1_1) at (0, 0) {$\bullet$};
\node (A1_2) at (3, 1.73) {$\bullet$};
\node (A1_3) at (3, -1.73) {$\bullet$};        \node (A1_4) at (2, 0) {$\bullet$};
\node (B1_2) at (-2, 0) {$\bullet$};    
    
\draw[thick] (A1_1) to (A1_2);
\draw[thick] (A1_1) to (A1_3);
\draw[thick] (A1_1) to (A1_4);
\draw[thick] (A1_2) to (A1_3);
\draw[thick] (A1_2) to (A1_4);
\draw[thick] (A1_3) to (A1_4);
\draw[thick] (A1_1) to (B1_2);
\end{tikzpicture}\]


Using the constraints imposed by Lemma~\ref{lem:vertex_pairs}, one can verify that we have 
\[
|\supp(u_1) \cap \supp(u_2) \cap \supp(v_j)| = 1
\]
under any embedding. We leave this verification to the reader, which completes the proof of this case.
\end{proof}

\begin{case}
         $k=1$.
    \end{case}

    \begin{proof}[Proof of Case]
We now assume $k=1$ and prove the desired statement.
 This implies that $\ell\geq 3$. Since $|\supp(u_1)| = 2$ and $|\supp(u_1) \cap \supp(v_j)| = 1$ for all $j$, the pigeonhole principle implies that, up to relabeling, there exist distinct $v_1$ and $v_2$ such that
\begin{equation}\label{eq:k=1l>2}
|\supp(u_1) \cap \supp(v_1) \cap \supp(v_2)| = 1.
\end{equation}
We will conclude the proof in this case by showing for arbitrary $j>2$, we have
\[
|\supp(u_1)\cap \supp(v_1)\cap \supp(v_2)\cap \supp(v_j)|=1.
\]
Again, by Conditions~\ref{cond:linking_distinct_components} and~\ref{cond:weight_three_vertices}, there exists $\beta \neq \alpha$ such that $v_j \in G_\beta$ but $v_{1}, v_2 \notin G_\beta$. By considering a vertex of weight two in $G_\beta$, we see that $\Gamma$ contains an induced subgraph of the form:
\[\begin{tikzpicture}[xscale=1,yscale=1,baseline={(0,0)}]

\node at (-2, .4) {$2$};
\node at (0, .4) {$3$};
\node at (3, 1.73+.4) {$3$};
\node at (3, -1.73-.4) {$3$};
\node at (2, .4) {$2$};

\node at (0, -.4) {$v_j$};
\node at (2, -.4) {$u_1$};
\node at (3.4, 1.73) {$v_1$};
\node at (3.4, -1.73) {$v_2$};
    
\node (A1_1) at (0, 0) {$\bullet$};
\node (A1_2) at (3, 1.73) {$\bullet$};
\node (A1_3) at (3, -1.73) {$\bullet$};        \node (A1_4) at (2, 0) {$\bullet$};
\node (B1_2) at (-2, 0) {$\bullet$};    
    
\draw[thick] (A1_1) to (A1_2);
\draw[thick] (A1_1) to (A1_3);
\draw[thick] (A1_1) to (A1_4);
\draw[thick] (A1_2) to (A1_3);
\draw[thick] (A1_2) to (A1_4);
\draw[thick] (A1_3) to (A1_4);
\draw[thick] (A1_1) to (B1_2);
\end{tikzpicture}\]
Using \eqref{eq:k=1l>2} and the constraints from Lemma~\ref{lem:vertex_pairs}, one quickly finds that
\[
|\supp(u_1) \cap \supp(v_1) \cap \supp(v_2) \cap \supp(v_j)| = 1
\]
must hold. Again, we leave this verification to the reader. This completes the proof for this case.
\end{proof}
Combining the two cases concludes the proof of the lemma.
\end{proof}

The following lemma is the main result of this section and will be used to prove the general case of Theorem~\ref{thm:main}.

\begin{lem}\label{lem:annoying_embedding_lemma}
Let $\Gamma = \bigcup_{\alpha=1}^n G_\alpha$ be a nice graph such that each $G_\alpha$ contains at least four vertices, including at least one vertex of weight two and at least two vertices of weight three. Then, for any embedding of $\Gamma$ into $\Z^N$, we have
    \[
    N\geq n+V
    \]
    where $V$ is the number of vertices in $\Gamma$.
\end{lem}
\begin{rem}
In fact, the hypotheses of Lemma~\ref{lem:annoying_embedding_lemma} are sufficient to show that the embedding of $\Gamma$ is unique up to automorphism of $\mathbb{Z}^N$; such a graph is referred to as \emph{rigid} in \cite[Definition~4.2]{Aceto-McCoy-Park:2022-1}. However, we will have no need for this stronger conclusion.
\end{rem}


\begin{proof}
\setcounter{claim}{0}
    The hypotheses of Lemma~\ref{lem:vertex_pairs} and Lemma~\ref{lem:general_clique} are satisfied so their conclusions apply.

    \begin{claim}\label{claim:non-empty_intersection}
       Let $v_1, \dots, v_k$ be a set of $k \geq 2$ distinct vertices satisfying
\[
\bigcap_{j=1}^k \supp(v_j) \neq \emptyset.
\]
    Then there is a unique $G_\alpha$ containing all of $v_1, \dots, v_k$, and
\[
\left|\bigcap_{j=1}^k \supp(v_j)\right| = 1.
\]
    \end{claim}
    \begin{proof}[Proof of Claim]
    Since $\bigcap_{j=1}^k \supp(v_j) \subset \supp(v_1) \cap \supp(v_2)$, Lemma~\ref{lem:vertex_pairs} implies that
$\left|\bigcap_{j=1}^k \supp(v_j)\right| = 1$.
Furthermore, Lemma~\ref{lem:vertex_pairs} implies that $v_1, \dots, v_k$ induce a complete subgraph of $\Gamma$. Since every edge of $\Gamma$ is contained in precisely one of the $G_\alpha$, it follows that all of $v_1, \dots, v_k$ must be contained in the same subgraph $G_\alpha$ of $\Gamma$.
    \end{proof}
    
    \begin{claim}\label{claim:multisupport}
        Up to automorphism of $\mathbb{Z}^N$, we may order the orthonormal basis $\{e_1, \dots, e_N\}$ so that, for any $1 \leq \alpha \leq n$, and any vertex $v$ we have $v \in G_\alpha$ if and only if $e_\alpha \in \supp(v)$.
    \end{claim}
    \begin{proof}[Proof of Claim]
    Lemma~\ref{lem:general_clique} implies that, for each $G_\alpha$, we have
$\left|\bigcap_{v \in G_\alpha} \supp(v)\right| = 1$.
Moreover, Claim~\ref{claim:non-empty_intersection} implies that for $\alpha \neq \beta$, the sets $\bigcap_{u \in G_\alpha} \supp(u)$ and $\bigcap_{u \in G_\beta} \supp(u)$ are disjoint. Thus, we may perform an automorphism of $\mathbb{Z}^N$ and assume that
$$\bigcap_{v \in G_\alpha} \supp(v) = \{e_\alpha\}$$
for each $1 \leq \alpha \leq n$. This implies the forward direction.

For the converse, suppose we have performed such an automorphism. If $v$ is a vertex with $e_\alpha \in \supp(v)$, then
\[
\supp(v) \cap \bigcap_{u \in G_\alpha} \supp(u) \neq \emptyset,
\]
and it follows from Claim~\ref{claim:non-empty_intersection} that $v \in G_\alpha$, as required.
    \end{proof}
    Assume now that we have performed an automorphism of $\mathbb{Z}^N$ so that the conclusion of Claim~\ref{claim:multisupport} holds. For each vertex $v \in \Gamma$, we observe that there is a unique basis vector $e_j$ with $j > n$ such that $e_j \in \supp(v)$, as follows. According to Claim~\ref{claim:multisupport}, the set $\supp(v) \cap \{e_1, \dots, e_n\}$ records which subgraphs $G_\alpha$ contain $v$. By Condition~\ref{cond:linking_distinct_components} (if $v$ has weight two) and Condition~\ref{cond:weight_three_vertices} (if $v$ has weight three), we see that $v$ is contained in exactly $w(v) - 1$ of the $G_\alpha$. Since $|\supp(v)| = w(v)$, this leaves precisely one basis vector $e_j \in \supp(v)$ with $j > n$.

    On the other hand, Claim~\ref{claim:non-empty_intersection} and Claim~\ref{claim:multisupport} together show that for any pair of distinct vertices $u, v \in \Gamma$, we have
\[
\supp(u) \cap \supp(v) \subset \{e_1, \dots, e_n\}.
\]
It follows that there are $V$ distinct basis vectors $e_j$ with $j > n$ that appear in the support of a (necessarily unique) vertex. In conclusion, we see that
\[
|\supp(\Gamma)| \geq n + V.
\]
Since $N \geq |\supp(\Gamma)|$, this gives the desired bound.
\end{proof}

 The following two lemmas will be used to deal with particular special cases of Theorem~\ref{thm:main}.

\begin{lem}\label{lem:crude_bound}
Let $\Gamma=\cup_{\alpha=1}^n G_\alpha$ be a nice graph such that each $G_\alpha$ contains at least two vertices of weight two and at least one vertex of weight three. Then, for any embedding of $\Gamma$ into $\Z^N$, we have
    \[N\geq n+V_2,\]
    where $V_2$ is the number of weight two vertices in $\Gamma$. 
\end{lem}
\begin{proof}
For each $G_\alpha$, let $G_\alpha^{(2)}$ denote the subgraph consisting of all vertices of weight two. Lemma~\ref{lem:2-graph_clique} implies that
\[
\left|\supp\left(G_\alpha^{(2)}\right)\right| = \left|G_\alpha^{(2)}\right| + 1.
\]
Moreover, Lemma~\ref{lem:vertex_pairs} implies that for $\alpha \neq \beta$, the supports of $G_\alpha^{(2)}$ and $G_\beta^{(2)}$ are disjoint. Therefore, we have
\[
N \geq \sum_{\alpha=1}^n \left|\supp\left(G_\alpha^{(2)}\right)\right| =n + V_2,
\]
as desired.
\end{proof}

\begin{lem}\label{lem:crude_bound2}
    Let $\Gamma=\cup_{\alpha=1}^n G_\alpha$ be a nice graph such that each $G_\alpha$ contains at least one vertex of weight two and at least two vertices of weight three. Then, for any embedding of $\Gamma$ into $\Z^N$, we have
    \[N\geq \frac{7n}{3}.\] 
\end{lem}
\begin{proof}
First, we choose $2n$ vertices from $\Gamma$, of which $n$ have weight two and $n$ have weight three. Let $u_1, \dots, u_n$ be distinct vertices of weight two, and let $v_1, \dots, v_n$ be distinct vertices of weight three, such that $u_\alpha, v_\alpha \in G_\alpha$ for each $1 \leq \alpha \leq n$. By Lemma~\ref{lem:vertex_pairs}, we have
\[
|\supp(u_1, \dots, u_n)| = 2n.
\]
Now each $v_j$ is adjacent to precisely two of the $u_i$. Thus, for every $v_j$, the support $\supp(v_j)$ contains a unique basis vector that is not contained in $\supp(u_1, \dots, u_n)$. For any pair of non-adjacent vertices $v_j$ and $v_{j'}$, we have that $\supp(v_j)$ and $\supp(v_{j'})$ are disjoint. Therefore, if $\{v_1, \dots, v_n\}$ contains a subset of $m$ vertices that are pairwise non-adjacent, then we have
\[
N \geq |\supp(\Gamma)| \geq 2n + m.
\]
Since each vertex in $\{v_1, \dots, v_n\}$ is adjacent to precisely two others in the set, the subgraph induced by the $v_j$ is a disjoint union of cycles on $n$ vertices. Therefore, there exists a pairwise non-adjacent subset of $\{v_1, \dots, v_n\}$ of size at least $n/3$. This completes the proof.
\end{proof}


\section{Proof of $n \geq k^2-1$}\label{sec:ngeqk^2-1}

In this section, we prove that $n \geq k^2-1$ for smooth
$\C$-realizations of $(n_k)$-configurations.

\begin{thm}\label{thm:k2minus1-bound}
If there exists a smooth $\C$-realization of an $(n_k)$-configuration with $k\geq 3$, then
\[
n \geq  k^2-1.
\]
\end{thm}

\begin{proof}

By \eqref{eq:n_k_inequality}, the only case of interest with $k=3$
is that of a $(7_3)$-configuration, while the only cases of interest with
$k=4$ are those of $(13_4)$- and $(14_4)$-configurations. The
$(7_3)$- and $(13_4)$-configurations attain equality in
\eqref{eq:n_k_inequality}; hence they are finite projective planes of
orders $2$ and $3$, respectively. These planes are unique~\cite{Batten:1997}
and are therefore $\mathbb{P}^2(\mathbb{F}_2)$ and
$\mathbb{P}^2(\mathbb{F}_3)$. Thus they are not topologically
$\C$-realizable by \cite[Theorem~1.1]{Ruberman-Starkston:2019}. The unique
$(14_4)$-configuration is not topologically $\C$-realizable by
\cite[Theorem~1.2]{Ruberman-Starkston:2019}. Thus we may assume $k\geq 5$.

Let $S_1, \dots, S_n \subset \mathbb{CP}^2$ be a collection of spheres forming a smooth $\mathbb{C}$-realization of an $(n_k)$-configuration. By \eqref{eq:n_k_inequality}, we may write
\[
n = (k^2 - k + 1) + d,
\]
where Theorem~\ref{thm:RS} shows that $d \geq 1$.

We aim to show that if $d \leq k - 3$, then no smooth $\mathbb{C}$-realization of the configuration exists, which in turn yields the desired bound. Thus we will assume $d\leq k-3$. In particular, we have
\[
k-d-2\geq 1.
\]

First, perform negative blow-ups on $\mathbb{CP}^2$ at each of the $k$-fold intersection points of the $S_\alpha$ corresponding to the subset $\mathcal{K}$ in the definition of an $(n_k)$-configuration. In addition, perform a smooth isotopy so that all remaining intersection points become transverse double points. Following these blow-ups, the proper transforms of $S_1, \dots, S_n$ form a collection of smooth spheres $S_1', \dots, S_n'$ in $\mathbb{CP}^2 \# n \, \overline{\mathbb{CP}}^2$. Since we blow up each $S_\alpha$ at $k$ points, each $S_\alpha'$ has self-intersection number $1 - k$. Moreover, because each blow-up on $S_\alpha$ removes its intersection with $k - 1$ other spheres, each $S_\alpha'$ intersects precisely $d$ other spheres, all transversely and at double points.


Next, perform positive blow-ups at each of the $nd/2$ intersection points between pairs of the $S_\alpha'$, and then perform an additional $(k - d - 2)$ positive blow-ups at distinct points on each $S_\alpha'$. This results in a total of
\[
N = \frac{nd}{2} + n(k - d - 2) = n\left(k - \frac{d}{2} - 2\right)
\]
blow-ups. Let $S_1'', \dots, S_n''$ denote the proper transforms of the $S_\alpha'$, and let $\Sigma_1, \dots, \Sigma_N$ denote the corresponding exceptional spheres. Then the $S_\alpha''$ and the $\Sigma_i$ are smoothly embedded spheres in
\[
(1 + N)\,\mathbb{CP}^2 \# n\,\overline{\mathbb{CP}}^2.
\]
The following demonstrates this process for a specific example with $n = 23$, $k = 5$, and $d = 2$, under the assumption that the combinatorial solution for the $(23_5)$-configuration under consideration yields a length-three cycle of the $S'_\alpha$ after the initial sequence of blow-ups. 

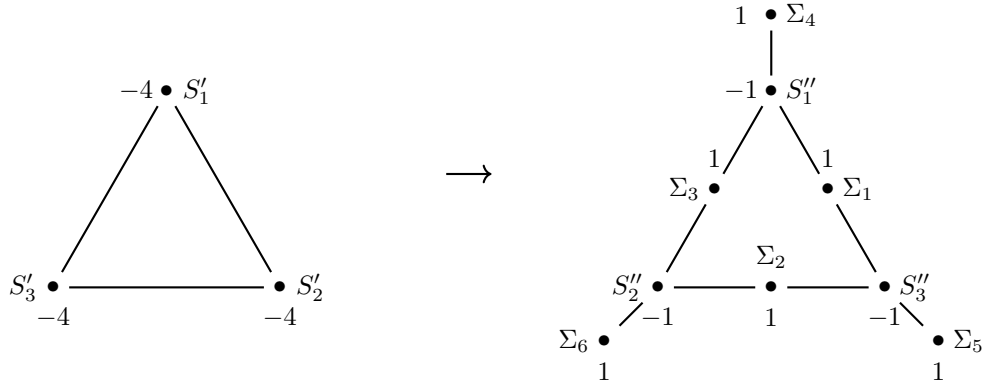
\[\begin{tikzpicture}[xscale=1,yscale=1,baseline={(0,0)}]
\node (A1_1) at (0, 0) {$\bullet$};
\node (A1_2) at (3, 0) {$\bullet$};
\node (A1_3) at (1.5, 2.6) {$\bullet$};

\node at (-.4, 0) {$S_3'$};
\node at (0, -.4) {$-4$};
\node at (3+.4, 0) {$S_2'$};
\node at (3, -.4) {$-4$};
\node at (1.5-.4, 2.6) {$-4$};
\node at (1.5+.4, 2.6) {$S_1'$};

\begin{scope}[shift={(1,0)}]
\node (B1_1) at (0+7, 0) {$\bullet$};
\node (B1_2) at (3+7, 0) {$\bullet$};
\node (B1_3) at (1.5+7, 2.6) {$\bullet$};

\node at (0+7-.4, 0) {$S_2''$};
\node at (0+7, 0-.4) {$-1$};
\node at (3+7+.4, 0) {$S_3''$};
\node at (3+7, 0-.4) {$-1$};
\node at (1.5+7+.4, 2.6) {$S_1''$};
\node at (1.5+7-.4, 2.6) {$-1$};

\node (C1_1) at ({17/2}, 0) {$\bullet$};
\node (C1_2) at ({4.5/2+7}, 1.3) {$\bullet$};
\node (C1_3) at ({1.5/2+7}, 1.3) {$\bullet$};

\node at ({17/2}, 0+.4) {$\Sigma_2$};
\node at ({17/2}, 0-.4) {$1$};
\node at ({4.5/2+7}, {1.3+.4}) {$1$};
\node at ({4.5/2+7+.4}, {1.3}) {$\Sigma_1$};
\node at ({1.5/2+7}, {1.3+.4}) {$1$};
\node at ({1.5/2+7-.4}, 1.3) {$\Sigma_3$};

\node (D1_1) at (1.5+7, 2.6+1.5-.5) {$\bullet$};
\node (D1_2) at (7-1.06+.35,-1.06+.35) {$\bullet$};
\node (D1_3) at (10+1.06-.35, 0-1.06+.35) {$\bullet$};

\node at (1.5+7-.4, 2.6+1.5-.5) {$1$};
\node at (1.5+7+.4, 2.6+1.5-.5) {$\Sigma_4$};
\node at (7-1.06-.4+.35,-1.06+.35) {$\Sigma_6$};
\node at (7-1.06+.35,-1.06-.4+.35) {$1$};
\node at (10+1.06+.4-.35, 0-1.06+.35) {$\Sigma_5$};
\node at (10+1.06-.35, 0-1.06-.4+.35) {$1$};

\node at (-1.06-.4+.35,-1.06+.35) {$\phantom{Sigma_6}$};
\node at (-1.06-.4+.35,-1.06+.35-.5) {$\phantom{Sigma_6}$};

\draw[thick] (A1_1) to (A1_2);
\draw[thick] (A1_1) to (A1_3);
\draw[thick] (A1_2) to (A1_3);

\draw[thick] (B1_1) to (C1_1);
\draw[thick] (B1_1) to (C1_3);
\draw[thick] (B1_2) to (C1_2);
\draw[thick] (B1_2) to (C1_1);
\draw[thick] (B1_3) to (C1_2);
\draw[thick] (B1_3) to (C1_3);
\draw[thick] (B1_3) to (D1_1);
\draw[thick] (B1_2) to (D1_3);
\draw[thick] (B1_1) to (D1_2);
\draw[thick, ->] (4.2,1.5) to (4.8,1.5);
\end{scope}
\end{tikzpicture}
\]

Since we have performed a total of $k - 2$ blow-ups on each $S_\alpha'$, the $S_\alpha''$ all have self-intersection $-1$. Furthermore, since we blew up at all intersection points between pairs of $S_\alpha'$, the $S_\alpha''$ are all disjoint. Thus, we may perform blow-downs on all the $S_\alpha''$ to convert $(1 + N)\,\mathbb{CP}^2 \# n\,\overline{\mathbb{CP}}^2$ into a smooth, positive-definite 4-manifold $X$ with $b_2(X) = 1 + N$. By Donaldson's diagonalization theorem, the intersection form on $\left(H_2(X;\Z)/ \tors ,Q_{X}\right)$ is equivalent to the standard diagonal lattice $\Z^{N+1}$. Let $\widetilde{\Sigma}_i$ denote the image of $\Sigma_i$ in $X$ after blowing down each of the $S_\alpha''$.

Define a graph $\Gamma$ as follows. Take the vertex set of $\Gamma$ to be the collection of $\Sigma_i$, with an edge between $\Sigma_i$ and $\Sigma_j$ if and only if there exists some $S_\alpha''$ intersecting both $\Sigma_i$ and $\Sigma_j$. We assign a weight of two to $\Sigma_i$ if it intersects precisely one of the $S_\alpha''$, and a weight of three if it intersects precisely two of the $S_\alpha''$.
This graph admits a decomposition $\Gamma = \bigcup_{\alpha = 1}^n G_\alpha$, where each $G_\alpha$ contains the $\Sigma_i$ that intersect $S_\alpha''$. By construction, each $G_\alpha$ contains $k - d - 2 \geq 1$ vertices of weight two and $d$ vertices of weight three. Moreover, $\Gamma = \bigcup_{\alpha = 1}^n G_\alpha$ is nice in the sense of Definition~\ref{def:nice_graph}.

Since the $\Sigma_i$ are pairwise disjoint, we see that two distinct $\widetilde{\Sigma}_i$ and $\widetilde{\Sigma}_j$ intersect if and only if $\Sigma_i$ and $\Sigma_j$ intersect the same sphere $S_\alpha''$, and that this intersection is a single transverse double point. Since each $\Sigma_i$ has self-intersection $+1$, it follows that $\widetilde{\Sigma}_i$ has self-intersection three if $\Sigma_i$ is the exceptional sphere from the blow-up of an intersection point between the $S_\alpha'$ spheres, and self-intersection two if $\Sigma_i$ arises from a blow-up on a single $S_\alpha'$ sphere.
Consequently, the homology classes represented by the $\widetilde{\Sigma}_i$ in the positive-definite 4-manifold $X$, whose intersection form is $\left(H_2(X;\Z)/ \tors ,Q_{X}\right) \cong \mathbb{Z}^{N+1}$, correspond to an embedding of the weighted graph $\Gamma$ into $\mathbb{Z}^{N+1}$.
The following continues the example with $n = 23$, $k = 5$, and $d = 2$, after blowing down each of the $S_\alpha''$. For the corresponding graph $\Gamma = \bigcup_{\alpha = 1}^n G_\alpha$, it illustrates the subgraphs $G_1$, $G_2$, and $G_3$, each of which contains $k - d - 2 = 1$ vertex of weight two and $d = 2$ vertices of weight three.

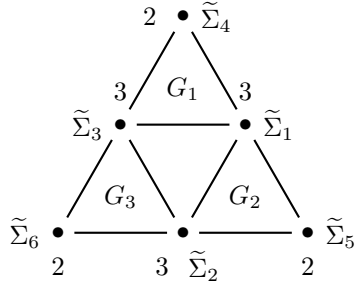
\[\begin{tikzpicture}[xscale=1.1,yscale=1.1,baseline={(0,0)}]
\node (B1_1) at (0+7, 0) {$\bullet$};
\node (B1_2) at (3+7, 0) {$\bullet$};
\node (B1_3) at (1.5+7, 2.6) {$\bullet$};  
\node at ({7+.75}, .43) {$G_3$};

\node at ({7+.75+1.5}, .43) {$G_2$};
\node at ({7+.75+.75}, .43+1.3) {$G_1$};

\node at (0+7-.4, 0) {$\widetilde{\Sigma}_6$};
\node at (0+7, 0-.4) {$2$};
\node at (3+7+.4, 0) {$\widetilde{\Sigma}_5$};
\node at (3+7, 0-.4) {$2$};
\node at (1.5+7+.4, 2.6) {$\widetilde{\Sigma}_4$};
\node at (1.5+7-.4, 2.6) {$2$};
\node at (0+7, 0-.6){$\phantom{Sigma_6}$};

\node (C1_1) at ({17/2}, 0) {$\bullet$};
\node (C1_2) at ({4.5/2+7}, 1.3) {$\bullet$};
\node (C1_3) at ({1.5/2+7}, 1.3) {$\bullet$};

\node at ({17/2+.25}, 0-.4) {$\widetilde{\Sigma}_2$};
\node at ({17/2-.25}, 0-.4) {$3$};
\node at ({4.5/2+7}, {1.3+.4}) {$3$};
\node at ({4.5/2+7+.4}, {1.3}) {$\widetilde{\Sigma}_1$};
\node at ({1.5/2+7}, {1.3+.4}) {$3$};
\node at ({1.5/2+7-.4}, 1.3) {$\widetilde{\Sigma}_3$};

\draw[thick] (B1_1) to (C1_1);
\draw[thick] (B1_1) to (C1_3);
\draw[thick] (B1_2) to (C1_2);
\draw[thick] (B1_2) to (C1_1);
\draw[thick] (B1_3) to (C1_2);
\draw[thick] (B1_3) to (C1_3);
\draw[thick] (C1_2) to (C1_1);
\draw[thick] (C1_3) to (C1_2);
\draw[thick] (C1_3) to (C1_1);

\end{tikzpicture}
\]

First, suppose that $d = 1$. In this case, $N = n\left(k - \frac{5}{2}\right)$, and each $G_\alpha$ contains $k - 2$ vertices, of which $k - 3$ have weight two and $d=1$ has weight three. Since we are assuming $k \geq 5$, it follows that each $G_\alpha$ contains at least two vertices of weight two. Thus, Lemma~\ref{lem:crude_bound} implies that an embedding of $\Gamma$ into $\mathbb{Z}^{N+1}$ exists only if
\[
N + 1 = n\left(k - \frac{5}{2}\right) + 1 \geq n + n(k - 3),
\]
which is absurd since $n > 2$. Hence, we may assume that $d \geq 2$.

If $k \geq 6$, then each $G_\alpha$ contains $k - 2 \geq 4$ vertices, of which $k - d - 2 \geq 1$ have weight two and $d \geq 2$ have weight three. Thus, by Lemma~\ref{lem:annoying_embedding_lemma}, an embedding of $\Gamma$ into $\mathbb{Z}^{N+1}$ can exist only if
\[
N + 1 \geq n + N,
\]
which is absurd since $n > 1$.

Thus, the remaining case is $k = 5$. In this case, by the assumption that $2 \leq d \leq k - 3$, we have $d = 2$, so that each $G_\alpha$ contains $k - 2 = 3$ vertices, of which $k - d - 2 = 1$ has weight two and $d = 2$ have weight three. Thus, Lemma~\ref{lem:crude_bound2} shows that $\Gamma$ can admit an embedding into $\mathbb{Z}^{N+1}$ only if
\[
N + 1 = 2n + 1 \geq \frac{7n}{3},
\]
which is absurd since $n = 23$ in this case.

Therefore, having eliminated all possibilities when $d \leq k - 3$, we conclude that a smooth $\mathbb{C}$-realization can occur only when $d \geq k - 2$, from which it follows that
\[
n = (k^2-k+1) + d \geq k^2 - 1,
\]
as required.
\end{proof}

\section{Proof of $n \neq k^2-1$}\label{sec:nneqk^2-1}

Let $\mathcal A=(\mathcal L,\mathcal P)$ be an $(n_k)$-configuration with distinguished set $\mathcal K\subset\mathcal P$. As before, write
\[
n=(k^2-k+1)+d.
\]
In this section, we consider the case
\[
n=k^2-1,
\qquad\text{equivalently}\qquad
d=k-2.
\]

\begin{defn}\label{def:residual}
Let $\mathcal A=(\mathcal L,\mathcal P)$ be an $(n_k)$-configuration with distinguished set $\mathcal K\subset\mathcal P$. We call a point $q\in\mathcal P\smallsetminus\mathcal K$ a \emph{residual point}. If $L,L'\in\mathcal L$ are distinct lines, let $P(L,L')$ denote the unique point of $\mathcal P$ that lies on both $L$ and $L'$. We call the two-element set $e=\{L,L'\}$ a \emph{residual pair} if $P(L,L')$ is a residual point.
\end{defn}

Let $\mathcal E$ be the set of residual pairs. Fix a line $L\in\mathcal L$. The $k$ points of $\mathcal K$ that lie on $L$ account for $k(k-1)$ lines other than $L$: for each such point there are $k-1$ other lines containing it, and no other line is counted twice, since two distinct lines have a unique common point. Hence exactly
\[
n-1-k(k-1)=d
\]
lines meet $L$ at residual points. Equivalently, $L$ belongs to exactly $d$ residual pairs. Double counting residual pairs by their two lines gives
\[
N:=|\mathcal E|=\frac{nd}{2}=\frac{n(k-2)}2.
\]
This agrees, in the borderline case $d=k-2$, with the number of positive blow-ups in the construction of Theorem~\ref{thm:k2minus1-bound}.

Now suppose, for contradiction, that $\mathcal A$ admits a smooth $\C$-realization. For each line $L\in\mathcal L$, let $S_L$ be the corresponding sphere, and let $H$ denote the pullback of the homology class of a complex projective line in $\CP^2$. 
Blow up negatively at each point $p\in\mathcal K$, and write $E_p$ for the corresponding negative exceptional class. After a small isotopy, for every residual pair $e=\{L,L'\}$ the proper transforms of $S_L$ and $S_{L'}$ meet in a transverse double point near $P(L,L')$. Blow up positively at all these double points, and write $F_e$ for the corresponding positive exceptional class. Since $d=k-2$, these are the only positive blow-ups needed.

Fix a choice of ordering of every residual pair $e=\{L,L'\}$ and define
\[
\epsilon_{L,e}=
\begin{cases}
  1, &\text{if }L\text{ is the first line in }e,\\
 -1, &\text{if }L\text{ is the second line in }e,\\
  0, &\text{if }L\notin e.
\end{cases}
\]
Thus the two nonzero signs in a column indexed by $e$ are opposite. In the ambient diagonal lattice with basis $H$, the $F_e$, and the $E_p$, the final proper transform of $S_L$ has class
\[
[S_L'']
=
H+\sum_{e\in\mathcal E}\epsilon_{L,e}F_e
-\sum_{\substack{p\in\mathcal K\\ p\text{ lies on }L}}E_p.
\]
These classes have square $1+d-k=-1$. Geometrically, the spheres $S_L''$ are pairwise disjoint in
\[
X:=(1+N)\CP^2\# n\,\overline{\CP}^2.
\]
Set
\[
L_{\mathcal A}
=
\left\langle [S_L'']\mid L\in\mathcal L\right\rangle^\perp
\subset \left(H_2(X;\Z)/ \tors ,Q_{X}\right).
\]
As in the proof of Theorem~\ref{thm:k2minus1-bound}, blowing down the pairwise disjoint $(-1)$-spheres $S_L''$ gives a smooth closed positive-definite 4-manifold whose intersection lattice is $L_{\mathcal A}$. Hence Donaldson's diagonalization theorem gives
\[
L_{\mathcal A}\cong \Z^{N+1}.
\]

Moreover, note that the vector
\[
u=kH-\sum_{p\in\mathcal K}E_p
\]
lies in $L_{\mathcal A}$. Indeed, for each line $L\in\mathcal L$,
\[
u\cdot [S_L'']
= k-\#\{p\in\mathcal K\mid p\text{ lies on }L\}
=0,
\]
because $L$ contains exactly $k$ points of $\mathcal K$. Moreover,
\[
u^2=k^2-n=k^2-(k^2-1)=1.
\]
Set
\[
M:=u^\perp\cap L_{\mathcal A}.
\]
Since $L_{\mathcal A}\cong\Z^{N+1}$ and $u$ is a unit vector, its
orthogonal complement is standard; hence
\[
M\cong\Z^N.
\]

First, we observe that for every vector in $M$, the coefficient of $H$ vanishes.

\begin{lem}\label{lem:balanced-equations}
Let $v\in M\cong\Z^N$, and write
\[
v=aH+\sum_{e\in\mathcal E} b_eF_e+\sum_{p\in\mathcal K}c_pE_p.
\]
Then
\[
a=0,\qquad \sum_{p\in\mathcal K}c_p=0,
\]
and, for every line $L\in\mathcal L$,
\begin{equation}\label{eq:line-balance}
\sum_{e\in\mathcal E}\epsilon_{L,e}b_e
+\sum_{\substack{p\in\mathcal K\\ p\text{ lies on }L}}c_p=0.
\end{equation}
\end{lem}

\begin{proof}
For each line $L\in\mathcal L$, the equation $v\cdot [S_L'']=0$ gives
\[
a+\sum_{e\in\mathcal E}\epsilon_{L,e}b_e
+\sum_{\substack{p\in\mathcal K\\ p\text{ lies on }L}}c_p=0
\]
for every line $L\in\mathcal L$. Summing these equations over all lines in $\mathcal L$ cancels the terms involving the $b_e$, since every residual pair contributes once with sign $1$ and once with sign $-1$. Since every point of $\mathcal K$ lies on exactly $k$ lines, this gives
\[
na+k\sum_{p\in\mathcal K}c_p=0.
\]
On the other hand, $v\in u^\perp$, so
\[
ka+\sum_{p\in\mathcal K}c_p=0.
\]
Comparing these two equations and using $n=k^2-1$, we have $a=0$, and then $\sum_p c_p=0$. Substituting $a=0$ in the orthogonality equations for each line $L\in\mathcal L$ gives \eqref{eq:line-balance}.
\end{proof}

Therefore, from now on, we can write each vector in $M\cong\Z^N$ as
\[
v=\sum_{e\in\mathcal E} b_eF_e+\sum_{p\in\mathcal K}c_pE_p.
\]
For a residual pair $e=\{L,L'\}$, define
\[
f_e:=F_e+\epsilon_{L,e}[S_L'']
        +\epsilon_{L',e}[S_{L'}''].
\]
The vectors $f_e$, for $e\in\mathcal E$, are the $N$ square-three vectors in
the orthogonal complement mentioned in the introduction.

\begin{lem}\label{lem:residual-vectors}
For every residual pair $e\in\mathcal E$, we have $f_e\in M\cong\Z^N$ and $f_e^2=3$. Moreover, for any vector
\[
v=\sum_{e\in\mathcal E} b_eF_e+\sum_{p\in\mathcal K}c_pE_p\in M,
\]
we have
\[
f_e\cdot v=b_e.
\]
\end{lem}

\begin{proof}
Since $F_e\cdot [S_L'']=\epsilon_{L,e}$ and $[S_L'']^2=-1$, we have
\[
f_e\cdot [S_L'']=0.
\]
The same calculation applies to $S_{L'}''$. If $J\in\mathcal L\smallsetminus\{L,L'\}$, then $F_e$ does not appear in $[S_J'']$, so the pairwise orthogonality of the classes $[S_I'']$ gives $f_e\cdot[S_J'']=0$. Hence $f_e\in L_{\mathcal A}$. Moreover, since $u\cdot f_e=0$, we have $f_e\in M$.

Using $\epsilon_{L,e}^2=\epsilon_{L',e}^2=1$, we get
\[
f_e^2=1+2+2-1-1=3.
\]
Finally, if $v\in M$, then $v$ is orthogonal to $S_L''$ and $S_{L'}''$, and hence
\[
f_e\cdot v=F_e\cdot v=b_e,
\]
which completes the proof.
\end{proof}




Now we are ready to prove the desired inequality for $k\geq5$.

\begin{lem}\label{lem:kge5-borderline}
If there exists a smooth $\C$-realization of an $(n_k)$-configuration with $k\geq5$, then
\[
n\geq k^2.
\]
\end{lem}

\begin{proof}
By Theorem~\ref{thm:k2minus1-bound}, it is enough to rule out the borderline case $n=k^2-1$. Assume, for contradiction, that a smooth realization exists in this borderline case. By Donaldson's diagonalization theorem, as observed above, we have $M\cong\Z^N$ with orthonormal basis $\{e_1,\dots,e_N\}$.

We first choose an index $i$ such that
\[
\#\{e\in\mathcal E\mid i\in\supp(f_e)\}\leq3.
\]
Such an index exists by an elementary argument. Indeed, since $f_e^2=3$, we have $|\supp(f_e)|=3$ for every residual pair $e\in\mathcal E$. Hence
\[
\sum_{i=1}^N
\#\{e\in\mathcal E\mid i\in\supp(f_e)\}
=
\sum_{e\in\mathcal E}|\supp(f_e)|
=3N.
\]
Therefore some index $i$ lies in the support of at most three of the vectors $f_e$.

Set $v=e_i$, and write
\[
v=\sum_{e\in\mathcal E} b_eF_e+\sum_{p\in\mathcal K}c_pE_p
\]
using Lemma~\ref{lem:balanced-equations}. In this setting \eqref{eq:line-balance} gives
\begin{equation}\label{eq:line-balance2}
\sum_{e\in\mathcal E}\epsilon_{L,e}b_e
+\sum_{\substack{p\in\mathcal K\\ p\text{ lies on }L}}c_p=0,
\end{equation}
for every $L\in\mathcal L$.
By Lemma~\ref{lem:residual-vectors}, for every residual pair $e\in\mathcal E$ we have
\[
b_e=f_e\cdot v.
\]
Since $v=e_i$, this is the $i$-th coordinate of $f_e$. Moreover, since $f_e^2=3$ in the standard lattice $M\cong\Z^N$, every coordinate of $f_e$ is $0$ or $\pm1$. Thus $b_e\in\{0,\pm1\}$. By the choice of $i$, the coefficient $b_e$ is nonzero for at most three residual pairs. Since $v^2=1$,
\[
1=\sum_{e\in\mathcal E} b_e^2-\sum_{p\in\mathcal K}c_p^2.
\]
It follows that
\[
\sum_{p\in\mathcal K}c_p^2\leq2.
\]
Together with Lemma~\ref{lem:balanced-equations}, which gives $\sum_{p\in\mathcal K}c_p=0$, this implies that either all the $c_p$ are zero, or the only nonzero coefficients among the $c_p$ are $1$ and $-1$, occurring at two distinct points of $\mathcal K$.

If all the $c_p$ are zero, then $\sum_e b_e^2=1$, so exactly one coefficient $b_e$ is nonzero. For either line in the residual pair $e$, equation~\eqref{eq:line-balance2} has a nonzero residual term and no $c_p$-term, a contradiction.

Thus there are distinct points $p,q\in\mathcal K$ such that the only nonzero coefficients among the $c_p$ occur at $p$ and $q$. Then $\sum_e b_e^2=3$, so the nonzero residual coefficients can appear in at most six line equations. On the other hand, the $c_p$-term in \eqref{eq:line-balance2} is nonzero on all lines that contain exactly one of $p$ and $q$. If no line contains both $p$ and $q$, there are $2k$ such lines. If some line contains both $p$ and $q$, then this line is unique, and the two contributions cancel on it; hence there are $2k-2$ such lines. In either case, for $k\geq5$ there are at least eight lines for which the $c_p$-term in \eqref{eq:line-balance2} is non-zero. However, the $b_e$ terms in \eqref{eq:line-balance2} can be non-zero for at most six lines. Thus, we obtain a contradiction for $k\geq 5$.
\end{proof}

For the rest of the section, we treat the case of $(15_4)$-configurations;
this final case will complete the proof of Theorem~\ref{thm:main}.

\begin{lem}\label{lem:15-4-matching}
In a $(15_4)$-configuration, let $p,q\in\mathcal K$ be distinct points, and suppose that a line $L\in\mathcal L$ contains both $p$ and $q$. Let $A_1,A_2,A_3$ be the other three lines containing $p$, and let $B_1,B_2,B_3$ be the other three lines containing $q$. Then there is no permutation $\sigma$ of $\{1,2,3\}$ such that each pair $\{A_i,B_{\sigma(i)}\}$ is a residual pair.
\end{lem}

\begin{proof}
Fix $p$. The four lines containing $p$ contain $4(4-1)=12$ points of $\mathcal K$ distinct from $p$. These points are all distinct, because two distinct lines have a unique common point. Since $|\mathcal K|=15$, exactly two points of $\mathcal K\smallsetminus\{p\}$ do not lie on any line containing $p$. The same statement holds with $q$ in place of $p$.

Among the thirteen points of $\mathcal K\smallsetminus\{p,q\}$, at most four fail to lie on a line containing $p$ or fail to lie on a line containing $q$. Thus at least nine of these thirteen points lie both on a line containing $p$ and on a line containing $q$. Two of them are the other points of $\mathcal K$ lying on $L$, so at least seven points of $\mathcal K$ not lying on $L$ have this property.

Each such point $r$ not lying on $L$ lies on a unique line among
$A_1,A_2,A_3$ and on a unique line among $B_1,B_2,B_3$. Hence $r$ is the
common point of a unique pair $\{A_i,B_j\}$, and distinct points give
distinct pairs. There are nine possible pairs $\{A_i,B_j\}$ in total, and
at least seven of them correspond to points in $\mathcal K$. Thus at most
two of the pairs $\{A_i,B_j\}$ are residual pairs. A permutation $\sigma$
for which all three pairs $\{A_i,B_{\sigma(i)}\}$ are residual pairs would
require three residual pairs, which is impossible.
\end{proof}

\begin{lem}\label{lem:15-4-borderline}
The borderline case $(15_4)$ is not smoothly $\C$-realizable.
\end{lem}

\begin{proof}
Assume a smooth $\C$-realization of a $(15_4)$-configuration exists. The borderline setup applies with $k=4$, $n=15$, and $N=15$. By Donaldson’s theorem and the observation following the definition of $M$, we have $M\cong\Z^{15}$. Choose an index $i$ as in the proof of Lemma~\ref{lem:kge5-borderline}, and set $v=e_i$.

The proof of Lemma~\ref{lem:kge5-borderline} applies up to the final count. Thus not all the $c_p$ are zero, and hence there are distinct points $p,q\in\mathcal K$ such that the only nonzero coefficients among the $c_p$ occur at $p$ and $q$. If no line contains both $p$ and $q$, then the $c_p$-term in \eqref{eq:line-balance} is nonzero in eight line equations, while the nonzero residual coefficients can appear in at most six line equations. This is impossible, so some line $L\in\mathcal L$ contains both $p$ and $q$.

Let $A_1,A_2,A_3$ be the other three lines containing $p$, and let $B_1,B_2,B_3$ be the other three lines containing $q$. The $c_p$-term in \eqref{eq:line-balance} is nonzero precisely on these six lines. Since $\sum_e b_e^2=3$ and each $b_e\in\{0,\pm1\}$, exactly three residual coefficients are nonzero. Equation~\eqref{eq:line-balance} forces the corresponding three residual
pairs with $b_e\ne0$ to involve exactly the six lines
\[
A_1,A_2,A_3,B_1,B_2,B_3,
\]
each appearing once. No residual pair can consist of two of the $A_i$, since
any two of them meet at the configuration point $p$. Similarly, no residual
pair can consist of two of the $B_j$. Hence the three residual pairs with
$b_e\ne0$ determine a permutation $\sigma$ such that every pair
$\{A_i,B_{\sigma(i)}\}$ is a residual pair. This contradicts
Lemma~\ref{lem:15-4-matching}.
\end{proof}

Finally, we prove the main theorem.

\begin{proof}[Proof of Theorem~\ref{thm:main}]
Suppose that there exists an $(n_k)$-configuration with $k\geq 4$ admitting a smooth $\C$-realization. By Theorem~\ref{thm:k2minus1-bound}, we have $n\geq k^2-1$. It remains to exclude the borderline case $n=k^2-1$. This is ruled out by Lemma~\ref{lem:kge5-borderline} when $k\geq 5$, and by Lemma~\ref{lem:15-4-borderline} when $k=4$. Hence $n\geq k^2$.
\end{proof}

\appendix

\section{A geometric line-arrangement bound}\label{app:geometric-bound}

For completeness, we record the following bound for geometric
$\C$-realizations, which follows from the classical inequalities for complex
line arrangements.

\begin{prop}\label{prop:hirzebruch}
If there exists a geometric $\C$-realization of an $(n_k)$-configuration with $k \geq 4$, then
\[
n\geq
\begin{cases}
16, & k=4,\\[2mm]
\left\lceil \dfrac32 k^2-3k+3\right\rceil, & k\geq 5.
\end{cases}
\]
\end{prop}

\begin{proof}
Let $\mathcal L$ be the corresponding arrangement of $n$ complex projective lines, and let $t_i$ denote the number of points of $\mathcal L$ with multiplicity $i$. Since $\mathcal L$ realizes an $(n_k)$-configuration, the $n$ points of $\mathcal K$ have multiplicity $k$, so $t_k\geq n$.

Assume first that $k\geq 5$. We claim that no singular point of the arrangement has multiplicity greater than $2n/3$. If $p\in\mathcal K$, then the
multiplicity of $p$ is $k$, and the elementary bound
$n\geq k^2-k+1$ implies $k\leq 2n/3$. If $p\notin\mathcal K$ has
multiplicity $m$, then the $m$ lines through $p$ contain pairwise disjoint
subsets of $\mathcal K$, each of size $k$. Hence $mk\leq n$, so
$m\leq n/k\leq 2n/3$.

Thus Langer's inequality applies: for an arrangement of $n$ complex projective lines with no point of multiplicity greater than $2n/3$, one has
\[
\sum_{r\geq 2} r t_r \geq \left\lceil \frac{n^2}{3}+n\right\rceil
\]
by~\cite[Proposition~11.3.1]{Langer:2003}. Combining this with the pair-counting identity
\[
\sum_{r\geq 2} r(r-1)t_r=n(n-1)
\]
gives
\begin{equation*}\label{eq:langer-hirzebruch-form}
t_2+\frac34t_3
\geq
n+\sum_{r\geq 5}\left(\frac{r^2}{4}-r\right)t_r.
\end{equation*}
Since $t_k\geq n$, it follows that
\[
t_2+t_3
\geq
n+n\left(\frac{k^2}{4}-k\right).
\]
Pair counting now gives
\[
\binom n2
=\sum_{r\geq2}\binom r2 t_r
\geq
t_2+t_3+\binom k2 t_k
\geq
n\left(1+\frac{k^2}{4}-k\right)+n\binom k2.
\]
After dividing by $n>0$, this is equivalent to
\[
n\geq \frac32 k^2-3k+3.
\]
Since $n$ is an integer, the claimed bound for $k\geq5$ follows.

It remains to consider $k=4$. We first observe that $t_n=t_{n-1}=0$.
Indeed, if a point $p$ had multiplicity $n$ or $n-1$, then some line
through $p$ would contain at most two singular points, contradicting the
fact that every line contains four points of $\mathcal K$. Hence the
Hirzebruch inequality~\cite{Hirzebruch:1983} applies and gives
\[
t_2+\frac34t_3
\geq
n+\sum_{i\geq5}(i-4)t_i.
\]
In particular,
\[
t_2+t_3
\geq
n+\sum_{i\geq4}(i-4)t_i
\geq n.
\]
Using $t_4\geq n$ and pair counting, we obtain
\[
\binom n2
=\sum_{i\geq2}\binom i2t_i
\geq
t_2+t_3+\binom42t_4
\geq 7n.
\]
Thus $n\geq15$.

Suppose, for contradiction, that $n=15$. Then all the preceding inequalities
are equalities. In particular,
\[
t_2=15,\qquad t_4=15,\qquad t_r=0 \quad \text{ for }r\notin\{2,4\}.
\] It turns out that this is impossible: by the uniqueness theorem of
Barthel--Hirzebruch--Höfer for line configurations with only double and
quadruple points satisfying the condition that each line contains exactly
two double points~\cite[Kapitel~3.1.G, pp.~110--113]{Barthel-Hirzebruch-Hoefer:1987-weighted},
the only possibilities are the triangle arrangement and the Hesse
arrangement, with $3$ and $12$ lines respectively; see
also~\cite[Theorem~27]{Roulleau:2026}. Therefore, $n\geq16$ when $k=4$.\end{proof}

It is noteworthy that, for $k=4$, this geometric bound coincides with the smooth bound of Theorem~\ref{thm:main}.

\bibliographystyle{amsalpha}
\def\MR#1{}
\bibliography{bib}

\end{document}